\documentclass[12pt,a4paper]{amsart}

\usepackage{amsfonts,amssymb}

\usepackage{amsthm}

\usepackage{amsmath}

\usepackage{verbatim}

\usepackage{color}

\usepackage{a4wide}

\theoremstyle{plain}

\newtheorem{theorem}{Theorem}[section]
\newtheorem{lemma}[theorem]{Lemma}
\newtheorem{proposition}[theorem]{Proposition}
\newtheorem{corollary}[theorem]{Corollary}
\newtheorem{Counter-example}[theorem]{Counter-example}
\newtheorem{remark}[theorem]{Remark}

\theoremstyle{definition}

\theoremstyle{remark}

\long\def\symbolfootnote[#1]#2{\begingroup\def\thefootnote{\fnsymbol{footnote}}
\footnote[#1]{#2}\endgroup}

\usepackage{amssymb}

\begin{document}

\def\Q{\mathbb Q}
\def\R{\mathbb R}
\def\N{\mathbb N}
\def\Z{\mathbb Z}
\def\C{\mathbb C}
\def\S{\mathbb S}
\def\L{\mathbb L}
\def\H{\mathbb H}
\def\K{\mathbb K}
\def\X{\mathbb X}
\def\Y{\mathbb Y}
\def\Z{\mathbb Z}
\def\E{\mathbb E}
\def\J{\mathbb J}
\def\I{\mathbb I}
\def\T{\mathbb T}
\def\H{\mathbb H}

\title{On the first eigenvalue of the Laplace operator \\ for compact spacelike submanifolds in Lorentz-Minkowski spacetime $\L^m$}

\author{Francisco J. Palomo$^*$}
\address{
  Departamento de Matem\'{a}tica
Aplicada, 
Universidad de M\'{a}laga,  29071-M\'{a}laga (Spain)}
\email{fjpalomo@ctima.uma.es}

\author{Alfonso Romero}
\address{
  Departamento de Geometr\'{i}a y Toplog\'{i}a, Universidad de Granada, 18071-Granada (Spain)}
\email{aromero@ugr.es}



\date{}

\symbolfootnote[ 0 ]{The first author is partially supported by Spanish MINECO and ERDF project MTM2016-78807-C2-2-P and the second one by Spanish MINECO and ERDF project MTM2016-78807-C2-1-P. \newline 
$^*$ Corresponding author.

}

\thispagestyle{empty}

\begin{abstract}
 For any compact spacelike submanifold $M$ of Lorentz-Minkowski spacetime $\L^m$,
 a family of upper bounds for the first eigenvalue of the Laplace operator is obtained. For each one of these inequalities, becoming an equality can be characterized in geometric terms. In particular, the eigenvalue achieves one of these upper bounds if and only if $M$ lies minimally in a hypersphere of a spacelike hyperplane. The inequalities are inspired by well-known work of Reilly \cite{Re}.
However, his technique cannot be applied to our case. Even more, the same Reilly upper bound does not work always for such a $M$, as shown by a family of counter-examples. So, a new technique, based on an integral formula on compact spacelike sections of the light cone in $\L^m$ is developed. The technique is genuine in our setting, that is, it cannot be extended to another semi-Euclidean spaces of higher index.

\end{abstract}

\maketitle

\markboth{}{}

\hyphenation{Lo-rent-zi-an}

\section{Introduction}
\noindent Let $M$ be $n$-dimensional compact submanifold in the $m$-dimensional Euclidean space $\E^m$. Inspired by \cite{BW},
Reilly found in \cite{Re} an optimal extrinsic inequality for the first nonzero eigenvalue $\lambda_1$ of the Laplacian of the induced metric on $M$ in terms of the square of the length of
the mean curvature $\|\mathbf{H}\|^2$ as follows,
\begin{equation}\label{reilly}
\lambda_{1}\leq n\frac{\displaystyle{\int}_{M}\|\mathbf{H}\|^2\,dV}{\mathrm{Vol}(M)}\,,
\end{equation}
where $\mathrm{Vol}(M)$ is the volume of $M$.
Moreover, the equality holds in (\ref{reilly}) if and only if $M$ lies minimally
in some hypersphere in $\E^m$.

Now consider an $n$-dimensional spacelike submanifold $M$ of the Lorentz-Minkowski spacetime $\L^m$.
That is, endowed with the induced metric $M$ is a Riemannian manifold. Assume $M$ is compact (then necessarily $m\geq n+2$).  Then the following question arises in a natural way.
\begin{quote}
{\it Does formula $(\ref{reilly})$ hold for any compact spacelike submanifold in $\L^m$?}
\end{quote}
The answer to this question is negative in general. In fact, we show a counter-example in Section 3.
Namely, given any $n\geq 1$, there exists an isometric immersion of the unit round $n$-dimensional sphere $\S^n$ in $\L^{n+2}$ with $\| \mathbf{H}\| ^2 \geq 0$ and such that
$$n\,\frac{\displaystyle{\int}_{\S^n}\|\mathbf{H}\|^2\,dV}{\mathrm{Vol}(\S^n)}< n=\lambda_{1}(\S^n).$$

For some special families of compact spacelike submanifolds $\psi: M^n \to \L^m$ formula (\ref{reilly}) holds whenever several extra hypotheses are assumed. First of all, inequality (\ref{reilly}) clearly holds if $\psi(M^n)$ lies in a spacelike affine hyperplane of $\L^m$. Now, consider such a $\psi$ satisfies $\psi(M^n)\subset \Pi$, where $\Pi$ is a lightlike affine hyperplane in $\L^m$. Without loss of generality, we may consider $\Pi$ defined by means of $x_{1}=x_{m}$. Then, we write down $\psi=(\psi_{1},\psi_{2},...,\psi_{m-1},\psi_{m})$ with $\psi_{1}=\psi_{m}$ and hence the mapping $\widetilde{\psi}:M^{n}\to \E^{m-2}$ defined by $\widetilde{\psi}=(\psi_{2},...,\psi_{m-1})$ is an immersion such that the metric induced from $\E^{m-2}$ via $\widetilde{\psi}$ agrees with the metric induced from $\L^m$ via $\psi$. Moreover, if $\mathbf{H}$ and $\widetilde{\mathbf{H}}$ are the mean curvature vector fields relative to $\psi$ and $\widetilde{\psi}$, respectively, then  $\|\mathbf{H}\|^2=\|\widetilde{\mathbf{H}}\|^2$ clearly. Thus, making use of (\ref{reilly}) for $\widetilde{\psi}$, we obtain (\ref{reilly}) for $\psi$. Note that in these two cases we have $\|\mathbf{H}\|^2\geq 0$. At this point, recall that no compact spacelike submanifold in $\L^m$ satisfies $\mathbf{H}=0$ (as in the Euclidean case). On the other hand, if we assume $\mathbf{H}_{p}\neq 0$ for all $p\in M$, then $\mathbf{H}$ cannot be causal everywhere \cite{AER}.

Another situation where inequality (\ref{reilly}) holds for certain compact spacelike submanifolds of $\L^m$ is the following. Let $\psi: M^2 \to \L^4$ be a spacelike surface such that $\psi(M^2)$ lies in a lightlike cone of $\L^4$. We show in \cite{PaRo} that the Gauss curvature of $M^2$ satisfies $K=\|\mathbf{H}\|^2$. When $M^2$ is compact then $M^2$ has the topology of the sphere $\S^2$, \cite{PaRo}. Therefore, the Gauss-Bonnet Theorem gives $\int_{M}\|\mathbf{H}\|^2 dA=4\pi$. Now, we have $
\lambda_{1}\leq 8\pi/\mathrm{Area}(M^2),
$ from the Hersch inequality, \cite{Hersch}. Therefore, we arrive to integral inequality (\ref{reilly}). Equality holds if and only if $M^2$ is totally umbilical in $\L^4$ \cite[Theorem 5.4]{PaRo}. We would like to point out that previous argument strongly depends on the dimension.

In view of the previous discussion, the following question emerges naturally.
\begin{quote}
{\it Is there an alternative to inequality $(\ref{reilly})$ for any compact spacelike submanifold in $\L^m$?
}
\end{quote}
Our main aim is then to look for {\it optimal} upper bounds for the first eigenvalue of the Laplace operator of a compact spacelike submanifold $M$ in $\L^m$. In the same philosophy of (\ref{reilly}), we search for upper bounds for $\lambda_1$ in terms of the mean curvature vector field and volume of $M$ and to characterize when the upper bound is attained.

It should be noticed that the technique in \cite{Re} does not work by serious reasons in our setting. In fact, a careful reading of \cite{Re} reveals three fundamental facts none one of them with a useful counterpart 
in our case. The first one is an averaging principle in \cite[Proposition 3]{Re} which gives an integral formula for the restriction of a quadratic form on the $(m-1)$-dimensional unit sphere $\S^{m-1}\subset \E^{m}$. 
The second fact is that the normal bundle of any submanifold in $\E^m$ is naturally endowed with positive definite metric. Finally, Cauchy-Schwarz inequality for vectors in $\E^{m}$ is used several times \cite[formula (7)]{Re}.
Now, neither De Sitter spacetime $\S^{m-1}_{1}$ nor unit hyperbolic space $\H^{m-1}$ (the two nondegenerate hypersurfaces in $\L^m$ consisting of the unit spacelike vectors and unit timelike vectors, respectively) is compact. Hence, it is imposible to state an averaging in any case. Moreover, the normal bundle of a spacelike submanifold of $\L^m$ with codimension at least $2$ has  Lorentzian signature. Finally, Cauchy-Schwarz inequality for vectors in $\L^m$ clearly does not hold. 
The first aim of this paper is to develop a new and suitable technique to avoid the mentioned difficulties. A key fact has been to replace the unit sphere $\S^{m-1}\subset \E^{m}$ by certain spherical sections in the lightlike cone of $\L^m$. On these spherical sections an averaging principle is given (Section 4).

The following results are typical examples of those we obtain in this paper.

\vspace{2mm}
\begin{quote}
{\it {\bf Proposition \ref{oldprincipal}}. For each unit timelike vector $a\in \L^{m}$, the first eigenvalue $\lambda_1$ of the Laplace operator for a compact $n$-dimensional spacelike submanifold $M$ in Lorentz-Minkowski spacetime $\L^m$ satisfies
$$
\hspace*{-3cm}\mathrm{(E)} \hspace*{3cm}\lambda_1 \leq n\,\frac{\displaystyle{\int}_{M}\Big(\|\mathbf{H}\|^2 +\langle \mathbf{H}, a \rangle^2 \Big)\,dV}{\mathrm{Vol}(M) +\frac{1}{n} \displaystyle{\int}_{M}\| a^{T}\|^2\, dV}\,,
$$
where $a^{T}\in \mathfrak{X}(M)$ is, at any point $p\in M$, the orthogonal projection of $a$ on $T_{p}M$. If $\psi:M\to \L^m$ has the center of gravity located at the origin, the equality in $\mathrm{(E)}$ holds if and only if there exists $\mu_{a}\in C^{\infty}(M)$ such that
$
\Delta \psi + \lambda_{1} \psi=\mu_{a} a.
$}
\end{quote}

\vspace{2mm}
\noindent Of course, the assumption on the spacelike immersion to discuss the equality in (E) is not a geometric restriction. In fact, it is easily achieved by means of a suitable traslation of the original immersion.
As a direct consequence of previous result we get the main Theorem of this paper (Theorem \ref{reprincipal}) with a clear geometric meaning.

\vspace{2mm}

\begin{quote}
{\it With the same notation as above, we have
$$
\hspace*{-5cm}\mathrm{(E^{*})} \hspace*{4cm}\lambda_1 \leq n\, \frac{\displaystyle{\int}_{M}\Big(\|\mathbf{H}\|^2 +\langle \mathbf{H}, a \rangle^2 \Big) \,dV}{\mathrm{Vol}(M)}.\hspace*{-1cm}
$$
The equality holds if and only if  $M$ factors through a spacelike affine hyperplane $\Pi$ orthogonal to $a\in \L^m$ and $M$ is minimal in some hypersphere in $\Pi$ with radius $\sqrt{n/\lambda_1}$.}
\end{quote}

\vspace{2mm}

\noindent As expected, the upper bound for $\lambda_1$ given in $\mathrm{(E^{*}})$ is bigger than the upper bound in (\ref{reilly}) for the case of compact submanifolds in an Euclidean space. In the very particular case that $\psi:M \to \L^{m}$ factors through a spacelike affine hyperplane  previous Theorem implies inequality (\ref{reilly}) as we known. Of course, formula (E*) cannot be deduced from  (\ref{reilly}) using the immersion $\pi_a \circ \psi : M \longrightarrow a^{\perp}$, where $\pi_a : \mathbb{L}^m \longrightarrow a^{\perp}$ is the orthogonal projection (see Remark \ref{end} for details). 

As a consequence of previous result we have
(Corollary \ref{noche}).
\vspace{2mm}
\begin{quote}
{\it If the first eigenvalue $\lambda_1$ of the Laplace operator of a compact spacelike $n$-dimensional spacelike submanifold  $M$ in $\L^{n+2}$ satisfies $$\lambda_{1}=n(\|\mathbf{H}\|^2+ \langle \mathbf{H}, a \rangle^2)$$ for some unit timelike vector $a\in \L^m$. Then $M$ is contained in a spacelike affine hyperplane orthogonal to $a\in \L^m$ as a round $n$-sphere of radius $\sqrt{n/\lambda_{1}}$.}
\end{quote}
\vspace{2mm}

Finally, we end this Section with several brief commentaries on the structure of this paper.  In Section 4 we introduce the spherical section $\S^{m-2}_a \subset\L^m$ relative to a unit timelike vector $a\in \L^m$. Each $\S^{m-2}_a $ is a compact spacelike submanifold of $\L^m$ isometric to the unit round sphere $\S^{m-2}$. In Lemma \ref{formula-integral}  we recall an averaging principle specific for $\S^{m-2}_a$ (see \cite[Lemma 3.4 (b)]{GPR}) . Section 5 is devoted to introduce and to discuss the notion of $\lambda_1$-test vector field on a compact spacelike submanifold $M$ in  $\L^m$. It also includes a general upper bound for $\lambda_1$ (Lemma \ref{main-lemma}).
Section 6 contains the previously quoted main results. These are achieved when we specialize Lemma \ref{main-lemma} for some suitable choices of $\lambda_{1}$-test vector fields.  As has been noticed, besides of the technique, there is a remarkable difference between our results and the extrinsic upper bound for $\lambda_1$ obtained by Reilly  \cite{Re}. Namely, every result in this paper shows a family of upper bounds for $\lambda_1$ parametrized on the unit timelike vectors in $\L^m$.

\section{Preliminaries}
\noindent Let $\L^{m}$ be the $m$-dimensional Lorentz-Minkowski spacetime, that is, $\L^{m}$ is $\R^{m}$ endowed
with the Lorentzian metric 
\begin{equation}\label{lorentzmetric}
\langle\; ,\; \rangle =
-dx_{1}^{2}+dx_{2}^{2}+\cdots+dx_{m}^{2},
\end{equation}
 where
$(x_{1},x_{1},\dots, x_{m})$ are the canonical coordinates of
$\R^{m}$. For every $v\in \L^m$, we write $\| v\|^{2}$ for $\langle v ,v \rangle$
although, of course, $\| v\|^{2}$ is not $\geq 0$, in general. Along this paper we assume $m\geq 3$. 

A smooth immersion $\psi:M^{n}\rightarrow \L^{m}$ of an
$n$-dimensional (connected) manifold $M^n$ is said to be
spacelike if the induced metric tensor via $\psi$ (denoted also by
$\langle\, ,\, \rangle$) is a Riemannian metric on $M^n$. In this
case, we call $M^n$ as a spacelike submanifold.

Let $\overline{\nabla}$ and $\nabla$ be the Levi-Civita
connections of  $\L^{m}$ and $M^n$, respectively. Let
$\nabla^{\perp}$ be the connection on the normal bundle. The Gauss
and Weingarten formulas are
$$\overline{\nabla}_X Y=\nabla_XY + \mathrm{II}(X,Y)
\, \quad \mathrm{and} \, \quad
\overline{\nabla}_X\xi=-A_{\xi}X+\nabla^{\perp}_X\,\xi,$$
for any $X,Y \in \mathfrak{X}(M^{n})$ and $\xi \in
\mathfrak{X}^{\perp}(M^{n})$, and where $\mathrm{II}$ denotes the
second fundamental form of $\psi$. As usual, we have agreed to ignore the differential of the map $\psi$.
The shape operator
corresponding to $\xi$, $A_{\xi}$, is related to $\mathrm{II}$ by
$$\langle A_{\xi}X, Y \rangle = \langle \mathrm{II}(X,Y), \xi
\rangle,$$
for all $X,Y \in \mathfrak{X}(M^{n})$.

The mean curvature vector field of $\psi$ is given by
$\mathbf{H}=\frac{1}{n}\,\mathrm{tr}_{_{\langle\; , \;
\rangle}}\mathrm{II},$ and it satisfies the Beltrami equation 
\begin{equation}\label{beltrami}
\Delta \psi=n \mathbf{H},
\end{equation} where the $i$-th component of $\Delta \psi$ is the Laplace operator $\Delta$ of $M$ applied to the $i$-th component of $\psi$, i.e., $\Delta \psi =(\Delta \psi_{1} ,\cdots , \Delta \psi_{m} )$. Moreover, $\Delta\| \psi\|^{2}= 2n+ 2n\langle \psi, \mathbf{H}\rangle$ and therefore, when $M$ is compact, we have the well-known Minkowski formula
\begin{equation}\label{1040}
\int_{M}(1+\langle \psi, \mathbf{H}\rangle)dV=0.
\end{equation}

\section{Counter-example}\label{example}
\noindent Let $\E^{m}$ be the $m$-dimensional Euclidean space that is, $\E^{m}$ is $\R^{m}$ endowed
with its usual Riemannian metric. We denote a point $(t,y)\in \R^m$ with $t\in \R$ and $y\in \R^{m-1}$. It is a direct computation that
$$
\Psi\colon\E^{m}\to \L^{m+1},\quad (t,y)\mapsto (\cosh (t), \sinh(t), y)
$$
is an isometric embedding.

Let us consider the $(n\geq 1)$-dimensional unit round sphere $\S^n \subset \E^{n+1}$ endowed with the usual induced metric. Thus, 
$\psi=\Psi\mid_{\S^{n}}$ is an isometric embedding of $\S^n$ into $\L^{n+2}$.
Now, the normal bundle of $\psi$ is spanned at every point $(t,y)\in \S^n$ by the following normal vector fields
$$\mathbf{N_{1}}(t,y)=(\cosh (t), \sinh(t), 0),\,\,\, \mathbf{N_{2}}(t,y)=(t\sinh (t), t\cosh(t), y),$$
which satisfy $\langle \mathbf{N_{1}}, \mathbf{N_{1}} \rangle=-1=-\langle \mathbf{N_{2}}, \mathbf{N_{2}} \rangle$ and $\langle \mathbf{N_{1}}, \mathbf{N_{2}} \rangle =0$.

The mean curvature vector field of $\psi$ is given by
\begin{equation}\label{contraejeH}
\mathbf{H}(t,y)=\frac{1-t^2}{n}\mathbf{N_{1}}(t,y)- \mathbf{N_{2}}(t,y)
\end{equation}
and so 
\begin{equation}\label{gr1}
\|\mathbf{H}\|^2 (t,y) =1-\frac{(1-t^2)^2}{n^2}. 
\end{equation}
Expression (\ref{contraejeH}) may be obtained in an alternatively way from the Beltrami equation (\ref{beltrami}). In fact, as previously denote by $(t,y)$ the restrictions to $\mathbb{S}^n$ of the
usual coordinates in $\mathbb{R}^{n+1}=\mathbb{R} \times \mathbb{R}^n$. We have $\Delta y = -n y$, and moreover it is not difficult to show that
$$
\Delta \cosh (t)=-n t\sinh(t)+(1-t^2)\cosh(t),
$$
$$\Delta \sinh (t)=-n t\cosh(t)+(1-t^2)\sinh(t).
$$
Collecting previous formulas we arrive again to the formula (\ref{contraejeH}).

From (\ref{gr1}), $\|\mathbf{H}\|^2< 1$ holds in $\S^{n}$ minus two antipodal points. Therefore, inequality (\ref{reilly}) does not hold for $\psi$. In fact,
we have that the quotient 
$$
n\frac{\displaystyle{\int}_{\S^n}\|\mathbf{H}\|^2\,dV}{\mathrm{Vol}(\S^n)}< n=\lambda_{1}(\S^n),
$$
where $\lambda_{1}(\S^n)$ denotes the first non vanishing eigenvalue of the Laplace operator of $\S^n$ (for a proof $n=\lambda_{1}(\S^n)$ see for instance \cite[Chapter II]{ChavelE}).

Even more, left hand side of previous inequality can be more precisely estimated by using the following result.
\begin{lemma}\label{formula-integral0} {\rm \cite[Lemma VII.3.1]{ChavelR}}
For any symmetric bilinear form $Q$ on $\R^{m}$ we have,
$$
\int_{v\in \S^{m-1}} Q(v,v)\, dV=\frac{\mathrm{Vol}(\S^{m-1})}{m} \,\mathrm{trace}(A_{Q}),
$$
where $A_{Q}$ is the operator of $\R^{m}$ defined by $\langle A_{Q}(u), v\rangle=Q(u,v)$ for all $u,v\in \R^{m}$ $($here $\langle\,\,,\,\, \rangle$ denotes the usual Riemannian metric of $\R^m$ $)$.
\end{lemma}
\noindent In fact, choose $Q(v,v)=t^{2}$ with $v=(t,y)\in \R^{n+1}$. Then previous Lemma gives
$$
\int_{\S^{n}} (1-t^2)\, dV=\frac{n}{n+1}\mathrm{Vol}(\S^{n}).
$$
Now the Cauchy-Schwarz inequality for integrals is called to get
$$
\int_{\S^{n}} (1-t^2)^2\, dV\geq \frac{n^2}{(n+1)^2}\mathrm{Vol}(\S^{n}).
$$
Therefore, 
$$
n\frac{\displaystyle{\int}_{\S^{n}}\|\mathbf{H}\|^2\,dV}{\mathrm{Vol}(\S^n)}\leq  n- \frac{n}{(n+1)^2}<n.
$$

\bigskip

\begin{remark}
{\rm
The isometric embedding $\psi:\S^n \to \L^{n+2}$
satisfies $\psi(\S^n)\subset \Sigma_{f}$ where $$\Sigma_f=\{(f(y),y): y\in \R^{n} \}$$
is the entire spacelike graph in $\L^{n+2}$ corresponding with $f(y)=\sqrt{1+y_{1}^2}$, and whose mean curvature satisfies $H^2=1/n^2$, (see for instance \cite{Trei}). Moreover, let us note that $\Sigma_f$ is isometric to the Euclidean space $\E^{n+1}$ via $\Psi$. In particular, $\psi(\S^n)$ is a geodesic sphere in $\Sigma_{f}$. This fact gives us that $\psi(\S^n)$ is unknotted in the sense that it is the boundary of an open $(n+1)$-ball in $\Sigma_f\subset \L^{n+2}$ \cite{Ko}. Therefore, there is no relationship between this topological notion and the fact that (\ref{reilly}) holds for any spacelike embedding of $\S^n$ in $\L^{n+2}$.
}
\end{remark}

\bigskip

We end this section pointing out that the construction of this counter-example can be generalized as follows. Let us take any unit spacelike curve
$
\alpha\colon I\subset \E^1 \to \L^{2},
$ 
such that $I$ is an open interval with $[-1, 1]\subset I$ (i.e., $\alpha$ is an isometric immersion).
From $\alpha$ we can define the isometric immersion
$$
\Psi_{\alpha}\colon I\times \E^{n}\to \L^{n+2}=\L^{2}\times \E^{n-1},\quad (t,y)\mapsto (\alpha(t), y).
$$
The map $\Psi_{\alpha}$ is a cylinder over the curve $\alpha$. In a similar way to the case below, we consider $\psi_{\alpha}:=\Psi_{\alpha}\mid \S^{n}$ and then compute that at every point $(t,y)\in \S^n$ the mean curvature vector field of $\psi_{\alpha}$ satisfies
\begin{equation}\label{gr2}
\|\mathbf{H}_{\psi_{\alpha}}\|^2 (t,y) =1+\frac{(1-t^2)^2}{n^2}\| \alpha''(t)\|^2,
\end{equation}
note that $\| \alpha'(t)\|^2=1$ implies $\| \alpha''(t)\|^2\leq 0$. Therfore, for every non-geodesic unit spacelike curve $
\alpha\colon I\to \L^{2}
$ we have
$$
n\frac{\displaystyle{\int}_{\S^n}\|\mathbf{H}_{\psi_{\alpha}}\|^2\,dV}{\mathrm{Vol}(\S^n)}< n=\lambda_{1}(\S^n).
$$

\section{Set up}
\noindent For each unit timelike vector $a\in \mathbb{L}^m$ (i.e. with $\langle a ,a \rangle=-1$), we define the spherical section in $\L^{m}$ relative to $a$ as
$$
\S_{a}^{m-2}=\{v\in \L^{m}: \langle v ,v \rangle=0,\,\,\, \langle a ,v \rangle=-1\},
$$
that is, $\S_{a}^{m-2}$ is the intersection of the light cone of $\L^{m}$ with the spacelike hyperplane given by $\langle a ,x \rangle=-1$.

It is not difficult to see that $\S_{a}^{m-2}$ is an $(m-2)-$dimensional compact spacelike submanifold isometric to the unit round sphere $\S^{m-2}$.
The spherical section $\S_{a}^{m-2}$ may be seen as the fiber of the trivial subbundle $C_{a}(\L^{m})=\L^{m}\times \S_{a}^{m-2}$ of the tangent bundle $T\L^{m}=\mathbb{L}^m \times \mathbb{R}^m$. In fact, $C_{a}(\L^{m})$ turns into a very particular case of the null congruence of a spacetime with respect to any of its timelike vector fields \cite{GPR}, \cite{Har}. Moreover, the  Sasaki metric on $T\L^{m}$, constructed from the Lorentzian metric of $\L^{m}$, induces on each slice $\{x\}\times \S_{a}^{m-2}$ the Riemannian metric of $\S_{a}^{m-2}\subset \L^m$ \cite[Proposition 2.3]{GPR}.

The key tool we will use here is the following integral formula (compare with \cite[Lemma VII.3.1]{ChavelR}).
\begin{lemma}\label{formula-integral} {\rm \cite[Lemma 3.4 (b)]{GPR}}
For any symmetric bilinear form $Q$ on $\L^{m}$ and any unit timelike vector $a\in \L^m$ we have,
$$
\int_{v\in \S_{a}^{m-2}} Q(v,v)\, dV_{a}=\frac{\mathrm{Vol}(\S^{m-2})}{m-1}\Big[ mQ(a,a)+ \mathrm{trace}(A_{Q})\Big],
$$
where $A_{Q}$ is the operator of $\L^{m}$ defined by $\langle A_{Q}(u), v\rangle=Q(u,v)$ for all $u,v\in \L^{m}$.
\end{lemma}

Let us recall the well-known Minimum Principle for the smallest positive eigenvalue $\lambda_{1}$ of the Laplace operator $\Delta$ of a compact Riemannian manifold $(M,g)$ \cite[p. 186]{BGM}. For every non-zero $C^1$ function $f:M\to \R$ with $\int_{M}fdV=0$, we have that
\begin{equation}\label{minimum}
\int_{M}\| \nabla f \|^2 dV \geq \lambda_{1}\int_{M}f^2 dV\,,
\end{equation}
where $\nabla$ denotes the gradient operator on $M$. The equality holds if and only if $f$ is an eigenfunction of $\Delta$ corresponding to $\lambda_1$, that is, $\Delta f +\lambda_{1}f=0.$

\section{$\lambda_1$-test vector fields}
\noindent Let $\psi : M^n \rightarrow \mathbb{L}^m$ be a compact spacelike submanifold
(hence $m\geq n+2$) and $W\in \mathfrak{X}_{\psi}(M)$ a fixed vector field along the immersion $\psi$. For every $v\in \L^m$, let us consider $F_{v}\in C^{\infty}(M)$ given by
\begin{equation}\label{func}
F_{v}: M \to \R, \quad F_{v}(x)=\langle v, W(x)\rangle.
\end{equation} 
The vector field $W\in \mathfrak{X}_{\psi}(M)$ is said to be a $\lambda_{1}$-{\it test vector field} when for every $v\in \L^m$  we have
\begin{equation}\label{functest}
\int_{x\in M} F_{v}(x)   dV=0.
\end{equation}

From the Beltrami equation (\ref{beltrami}), the mean curvature vector field $\mathbf{H}$ is always a $\lambda_{1}$-test vector field. On the other hand, from any $W\in \mathfrak{X}_{\psi}(M)$, $W=(W_{1},..., W_{m})$, one can arrive to the $\lambda_{1}$-test vector field 
\begin{equation}\label{hat}
\widehat{W}=W-\frac{1}{\mathrm{Vol}(M)}c
\end{equation}
where $c=(c_{1},..., c_{m})$ and $c_{j}=\int_{M} W_{j} \, dV$, $1\leq j\leq m$.
Note that if $W$ is the restriction to $M^n$ of a fixed vector of $\L^m$ then $\widehat{W}=0$.

Let us fix a $\lambda_{1}$-test vector field $W$, using (\ref{functest}), the Minimum Principle (\ref{minimum}) gives 
\begin{equation}\label{f10}
\int_{x\in M}\|\nabla F_{v}\|^2(x)\, dV\geq \lambda_{1} \int_{x\in M} F_{v}^2(x)    dV,
\end{equation}
for all $v\in \L^m$. Now, let us fix a unit timelike vector $a\in \L^{m}$ and integrating both sides of (\ref{f10}) on $\S^{m-2}_a$, we get,
$$
\int_{v\in \S_{a}^{m-2}}\Big[\int_{x\in M}\|\nabla F_{v}\|^2 (x)   dV\Big]dV_{a}\geq \lambda_{1}\int_{v\in \S_{a}^{m-2}}\Big[\int_{x\in M} F_{v}^2(x)    dV\Big]dV_{a}.
$$
Then, we can make use of Fubini's Theorem to obtain
\begin{equation}\label{f20}
\int_{x\in M}\Big[\int_{v\in \S_{a}^{m-2}}\|\nabla F_{v}\|^2(x)    dV_{a}\Big]dV \geq \lambda_{1}\int_{x\in M}\Big[\int_{v\in \S_{a}^{m-2}} F_{v}^2(x)    dV_{a}\Big]dV.
\end{equation}
Next, the integral formula in Lemma \ref{formula-integral} is applied to the symmetric bilinear form $Q_{1}^{x}(u,v)=\langle \nabla F_{u}, \nabla F_{v}\rangle (x)$, $x\in M$ fixed, to
obtain
\begin{equation}\label{f30}
\int_{v\in \S_{a}^{m-2}}\|\nabla F_{v}\|^2 (x)  \, dV_{a}=
\frac{\mathrm{Vol}(\S^{m-2})}{m-1}\Big[ m\| \nabla F_{a} \|^2 (x)+ \mathrm{trace}(A_{Q_{1}^{x}}) \Big]
\end{equation}
and also to the symmetric bilinear form $Q_{2}^{x}(u,v)= F_{u}(x) F_{v} (x)$, $x\in M$,  to
obtain
\begin{equation}\label{f40}
\int_{v\in \S_{a}^{m-2}} F_{v}^2(x) \,  dV_{a}=\frac{\mathrm{Vol}(\S^{m-2})}{m-1}\Big[ m F_{a}^2 (x)+ \mathrm{trace}(A_{Q_{2}^{x}}) \Big].
\end{equation}
We easily see that
$\mathrm{trace}(A_{Q_{2}^{x}})=\| W\|^2(x)$ and clearly, $\Delta F_{v}+ \lambda_{1}F_{v}=0$ holds for all $v\in \S^{m-2}_a$ if and only if $\Delta W + \lambda_{1}W=0$.

Thus, we substitute (\ref{f30}) and (\ref{f40}) into inequality (\ref{f20}) to get the main technical result.

\begin{lemma}\label{main-lemma}
Let $\psi: M^{n}\to \L^{m}$ be a compact spacelike submanifold and $W\in \mathfrak{X}_{\psi}(M)$ a $\lambda_{1}$-test vector field. Then, for every unit timelike vector $a\in \L^m$ we have
\begin{equation}\label{tecnico}
\int_{ M}\Big[m\| \nabla \langle a, W\rangle \|^2 + \mathrm{trace}(A_{Q_{1}}) \Big]dV \geq \lambda_{1}\int_{M}\Big[ m \langle a, W\rangle ^2 + \| W\|^2 \Big]dV,
\end{equation}
where $\langle A_{Q_{1}}(u), v\rangle=\langle \nabla \langle u, W \rangle,\nabla \langle v, W \rangle \rangle$ for all $u,v\in \L^{m}$.
The equality holds if and only if we have $\Delta W + \lambda_{1}W =0.$
Thus, if equality holds for some $a$ then it holds for any unit timelike vector in $\L^m$.
\end{lemma}

In order to obtain a formula for $\mathrm{trace}(A_{Q_{1}^{x}})$, we summarize here several definitions. For $Z\in \mathfrak{X}(M)$, let us recall that $\| \nabla Z \|^2(x)=\sum_{i=1}^{n}\| \nabla_{e_{i}}Z\|^2 $ where $\{e_{1},...,e_{n}\}$ is an orthonormal basis of $T_{x}M$ and $\| AZ\|^2(x)= \sum_{j=n+1}^{m}\varepsilon_{j}\| A_{e_{j}}Z\|^2$ where $\{e_{n+1},...,e_{m}\}$ is an orthonormal basis of $T_{x}^{\perp}M$ with $\varepsilon_{j}=\langle e_{j}, e_{j}\rangle$. In a similar way for $\xi \in \mathfrak{X}^{\perp}(M)$, we define $\| \nabla^{\perp} \xi \|^2(x)=\sum_{i=1}^{n}\| \nabla^{\perp}_{e_{i}}\xi\|^2 $. 
The decomposition $W=W^{T}+ W^{\perp}$ where $W^{T}\in \mathfrak{X}(M)$ and $W^{\perp}\in \mathfrak{X}^{\perp}(M)$ for $W\in \mathfrak{X}_{\psi}(M)$ will be extensively used. Let us recall that for the particular case $v\in \L^{m}$, we have $v^{T}=\nabla \langle v, \psi\rangle$.

Assume that $\{e_{1},...,e_{n}\}$ are eigenvectors for $A_{W^{\perp}}$.
Now, we compute
$$
\mathrm{trace}(A_{Q_{1}^{x}})=\sum_{i=1}^{n}\| \nabla F_{e_{i}}\|^{2}(x)+\sum_{j=n+1}^{m}\varepsilon_{j}\| \nabla F_{e_{j}}\|^{2}(x).
$$
On the one hand, we have
\[
\sum_{i=1}^{n}\| \nabla F_{e_{i}}\|^{2}=\sum_{i,k=1}^{n}\langle \nabla F_{e_{i}}, e_{k}\rangle^2=\sum_{i,k=1}^{n}[e_{k}\langle e_{i}, W\rangle]^{2}
\]
\[
\qquad \qquad \qquad =\sum_{i,k=1}^{n}\Big[\langle e_{i}, \nabla_{e_{k}}W^{T}\rangle  -\langle e_{i}, A_{W^{\perp}}(e_{k}) \rangle \Big] ^{2}
\]
\[
\qquad \qquad \quad \qquad \qquad \quad \qquad \quad \quad=\| \nabla W^{T}\| ^{2}(x)+\mathrm{trace}(A^2_{W^{\perp}})(x)-2\mathrm{trace}(A_{W^{\perp}}\nabla W^{T})(x),
\]
where $A_{W^{\perp}}\nabla W^{T}$ is the endomorphism field on $M$ given by $A_{W^{\perp}}\nabla W^{T} (v)=A_{W^{\perp}}\nabla_{v} W^{T}$ for $v\in T_{x}M$.
On the other hand, in a similar way, we have
$$
\sum_{j=n+1}^{m}\varepsilon_{j}\| \nabla F_{e_{j}}\|^{2}=\| \nabla^{\perp}W^{\perp}\| ^{2}(x)+\| AW^{T}\| ^2 (x)+2\mathrm{trace}(A_{\nabla^{\perp}W^{\perp}}W^{T})(x),
$$
where $A_{\nabla^{\perp}W^{\perp}}W^{T}(v)=A_{\nabla_{v}^{\perp}W^{\perp}}W^{T}$.
Therefore we deduce tha following general formula
$$
\mathrm{trace}(A_{Q_{1}^x})
=\| \nabla W^{T}\| ^{2}(x)+ \| AW^{T}\| ^2 (x)+\mathrm{trace}(A^2_{W^{\perp}})(x)+\| \nabla^{\perp}W^{\perp}\| ^{2}(x)
$$
\begin{equation}\label{f45}
\qquad + 2\mathrm{trace}(A_{\nabla^{\perp}W^{\perp}}W^{T})(x)-2\mathrm{trace}(A_{W^{\perp}}\nabla W^{T})(x).
\end{equation}
In the particular case $\xi\in \mathfrak{X}^{\perp}(M)$, formula (\ref{f45}) reduces to
\begin{equation}\label{f50}
\mathrm{trace}(A_{Q_{1}^{x}})=\mathrm{trace}(A^2_{\xi})(x)+ \| \nabla^{\perp}\xi\| ^{2}(x)
\end{equation}
and for $Z\in \mathfrak{X}(M)$ we have
$
\mathrm{trace}(A_{Q_{1}^{x}})=\| \nabla Z\| ^{2}(x)+ \| AZ\| ^2 (x).
$

\bigskip

\begin{remark}\label{Wnon-zero}
{\rm The right hand side of inequality (\ref{tecnico}) never vanishes except for $W=0$ at every point of $M$. In fact, the unit timelike vector $a\in \L^m$ gives the orthogonal decomposition $\L^{m}=\mathrm{Span}(a)\oplus a^{\perp}$ and therefore $W$ has a unique expression $W=W_{a}-\langle W, a\rangle a$ attending to this decomposition. Thus, for the right hand side of (\ref{tecnico}) we have
$$
\int_{M}\Big[ m \langle a, W\rangle ^2 + \| W\|^2 \Big]dV=\int_{M}\Big[( m-1) \langle a, W\rangle ^2 + \| W_{a}\|^2 \Big]dV\geq 0,
$$
and the equality only holds for $W=0$ identically.
}
\end{remark}
\begin{remark}
{\rm 
Asume that  $\psi: M^{n}\to \L^{m}$ is a compact spacelike submanifold. For every non-zero smooth function $f:M\to \R$ with $\int_{M}fdV=0$, the vector field $\xi=f\cdot a \in \mathfrak{X}_{\psi}(M)$ is a $\lambda_1$-test vector field. A direct computation gives $\nabla F_{v} =\langle v, a\rangle \nabla f$ for $v\in \L^m$, and therefore $\mathrm{trace}(A_{Q_{1}})=-\| \nabla f \|^2$. Thus, in this case, Lemma \ref{main-lemma} reduces to the Minimum Principle (\ref{minimum}).}
\end{remark}

\section{Main results}
\noindent In this Section, we specialize previous Lemma \ref{main-lemma} for several choices of the $\lambda_{1}$-test vector field $W\in \mathfrak{X}_{\psi}(M)$. First, let us take $W=\mathbf{H}$.
\begin{proposition}\label{H1}
For every unit timelike vector $a\in \L^m$, the first eigenvalue $\lambda_1$ of the Laplace operator for a compact $n$-dimensional spacelike submanifold $M$ in Lorentz-Minkowski spacetime $\L^m$ satisfies
\begin{equation}\label{f60}
\lambda_{1}\leq \frac{
\displaystyle{\int}_{ M} \Big[ m\| \nabla \langle a, \mathbf{H}\rangle \|^2 + \mathrm{trace}(A^2_{\mathbf{H}})+ \| \nabla^{\perp}\mathbf{H}\| ^{2} \Big]dV}{\displaystyle{\int}_{M}\Big[ m \langle a, \mathbf{H}\rangle ^2 + \| \mathbf{H}\|^2 \Big]dV}.
\end{equation}
The equality holds for some unit timelike vector $a\in \L^m$ if and only if $M$ is inmersed in a De Sitter space of radius $\sqrt{n/\lambda_1}$ with zero mean curvature vector field. In particular, for $m=n+2$, the equality holds if and only if $M$ is a totally geodesic submanifold in a De Sitter space of radius $\sqrt{n/\lambda_1}$.
\end{proposition}
\begin{proof}
The inequality (\ref{f60}) is a direct consequence of Lemma \ref{main-lemma} and formula (\ref{f50}). Recall at this point that there is no compact spacelike submanifold in $\L^m$ with $\mathbf{H}=0$ and therefore Remark \ref{Wnon-zero} may be applied. If the equality holds in (\ref{f60}), then Lemma \ref{main-lemma} may be called, and by using the Beltrami equation (\ref{beltrami}) we have 
$$
\Delta \mathbf{H} +\lambda_{1}\mathbf{H}=\frac{1}{n}\Delta \Big[\Delta \psi+ \lambda_{1}\psi \Big]=0.
$$
Taking into account that $M$ is compact, we arrive to $\Delta \psi+ \lambda_{1}\psi=b\in \L^m$.

Let us consider now $\widehat{\psi}=\psi-b/\lambda_{1} $, that is, $\widehat{\psi}$ is the translation of $\psi$ by the vector $-b/\lambda_{1}$. Thus, we have $\Delta \widehat{\psi}+ \lambda_{1}\widehat{\psi}=0.$
Then, from the Semi-Riemannian version of the Takahashi result \cite[Theorem 1]{Mark}, one deduces that $\psi$ realizes an inmersion with zero mean curvature vector field in the De Sitter space of radius $\sqrt{n/\lambda_1}$ and center located at $b/\lambda_{1}$ in $\L^m$. Conversely, if $M$ is a spacelike submanifold in a De Sitter space of radius $\sqrt{n/\lambda_1}$, with zero mean curvature vector field, then \cite[Theorem 1]{Mark} also applies to obtain that $\Delta \psi+ \lambda_{1}\psi=0$ (up to possibly a parallel displacement). Therefore, $\Delta \mathbf{H} +\lambda_{1}\mathbf{H}=0$ holds and the proof ends using Lemma \ref{main-lemma}. In the particular case $m=n+2$, the assertion is a direct application of \cite[Theorem 1.1]{Ishi}. 
\end{proof}

\begin{remark}
{\rm 
If we particularize previous Proposition for the case of a compact \linebreak $n$-dimensional spacelike submanifold $M$ in Lorentz-Minkowski spacetime $\L^m$ through a spacelike hyperplane $\Pi$, we get
\begin{equation}\label{f70}
\lambda_{1}\leq \frac{
\displaystyle{\int}_{ M} \Big[ \mathrm{trace}(A^2_{\mathbf{H}})+ \| \nabla^{\perp}\mathbf{H}\| ^{2} \Big]dV}{\displaystyle{\int}_{M}\| \mathbf{H}\|^2 dV}.
\end{equation}
The equality holds if and only if $M$ is a minimal submanifold in some hypersphere in $\Pi$ of radius $\sqrt{n/\lambda_1}$.
Actually, (\ref{f70}) gives an upper bound for the first eigenvalue of the Laplace operator for compact submanifolds in Euclidean spaces. In order to compare (\ref{f70}) with Reilly inequality (\ref{reilly}), we consider the following string of inequalities
$$
\mathrm{trace}(A^2_{\mathbf{H}}) \geq \frac{1}{n}(\mathrm{trace}A_{\mathbf{H}})^2 \geq n\| \mathbf{H}\|^4.
$$
Now, the Cauchy-Schwarz inequality for integrals gives, 
$$
 \frac{
\displaystyle{\int}_{ M} \Big[ \mathrm{trace}(A^2_{\mathbf{H}})+ \| \nabla^{\perp}\mathbf{H}\| ^{2} \Big]dV}{\displaystyle{\int}_{M}\| \mathbf{H}\|^2 dV} \geq n \frac{
\displaystyle{\int}_{ M} \| \mathbf{H}\| ^{4}dV}{\displaystyle{\int}_{M}\| \mathbf{H}\|^2 dV}
\geq n\frac{\Big(\displaystyle{\int}_{M}\|\mathbf{H}\|^2\,dV\Big)^2}{\mathrm{Vol}(M)\displaystyle{\int}_{M}\|\mathbf{H}\|^2\,dV}=
n\frac{\displaystyle{\int}_{M}\|\mathbf{H}\|^2\,dV}{\mathrm{Vol}(M)}.
$$
Therefore, integral inequality (\ref{f70}) is weaker than (\ref{reilly}), in general. Moreover, the inequality (\ref{f70}) is just inequality (\ref{reilly}) if and only $M^n$ is a totally umbilical round sphere in a spacelike affine hyperplane of $\L^m$ \cite[Theorem 4.3]{AER}.
}
\end{remark}

Next assume that the center of gravity of the compact spacelike immersion $\psi: M^{n}\to \L^{m}$ is located at the origin. That is, the $j$-th component of $\psi$ satisfies $\int_{M}\psi_{j}dV=0$ for all $j=1,...,m$. Thus, the immersion $\psi$ is a $\lambda_{1}$-test vector field. Under this assumption, for $W=\psi$, with notation as in previous Section, we have for every $v\in \L^m$,
$$
\|\nabla F_{v}\|^2 (x)=\sum_{i=1}^{n} \Big( e_{i} \langle v , \psi\rangle  \Big)^{2}=
\sum_{i=1}^{n} \Big( \langle v , \overline{\nabla}_{e_{i}} \psi\rangle  \Big)^{2}=\sum_{i=1}^{n} \langle v ,e_{i} \rangle  ^{2}=\| v^{T}\|^2,
$$
where  $\{e_{1},...,e_{n}\}$ is an orthonormal basis of $T_{x}M$
and therefore
\begin{equation}\label{f2020}
\mathrm{trace}(A_{Q_{1}^{x}})=n
\end{equation}  for every $ x\in M.$
We are in a position to state,

\begin{lemma}\label{lemma2030}
The first eigenvalue $\lambda_1$ of the Laplace operator for a compact $n$-dimensional spacelike submanifold $M$ in Lorentz-Minkowski spacetime $\L^m$, with gravity center located at the origin, satisfies
\begin{equation}\label{f80}
\lambda_{1} \int_{M}\Big[ m \langle a, \psi \rangle ^2 + \| \psi \|^2 \Big]dV \leq n\mathrm{Vol}(M) +
m\int_{ M} \|  a^{T} \|^2 dV
\end{equation}
for every unit timelike vector $a\in \L^m$.
The equality holds for some $a\in \L^m$ (and then it holds for any unit timelike vector in $\L^m$) if and only if $M$ is inmersed in a De Sitter space of radius $\sqrt{n/\lambda_1}$ with zero mean curvature vector field.
\end{lemma}
\begin{proof} Taking into account that $a^{T}=\nabla \langle a ,\psi \rangle$ and (\ref{f2020}), Lemma \ref{main-lemma} implies inequality (\ref{f80}) with equality if and  only if $\Delta \psi +\lambda_{1}\psi=0$. Now, semi-Riemannian version of the Takahashi result \cite[Theorem 1]{Mark} can be again claimed to deduce that the equation $\Delta \psi +\lambda_{1}\psi=0$ is satisfied if and only if $\psi$ realizes a spacelike immersion with zero mean curvature vector field in the De Sitter space of radius $\sqrt{n/\lambda_1}$. 
\end{proof}

Next, we derive a family of $\lambda_1$-test vector fields from each compact spacelike compact immersion $\psi: M^{n}\to \L^{m}$ as follows. For every unit timelike vector $a\in \L^m$, consider 
\begin{equation}\label{f2010}
\psi_{a}:=\psi+\langle \psi, a \rangle a\in \mathfrak{X}_{\psi}(M),
\end{equation} 
that is, $\psi_a$ is the orthogonal projection of $\psi$ on the spacelike hyperplane $\E^{m-1}_{a}:=a^{\perp}$. 
Assume that the center of gravity of $\psi$ is located at the origin, then every $\psi_{a}$ is also a $\lambda_{1}$-test vector field. 
In the terminology of previous Section, for $W=\psi_{a}$ and $\{e_{1},...,e_{n}\}$ an orthonormal basis of $T_{x}M$, we have 
$$\|\nabla F_{v}\|^2 (x)=\sum_{i=1}^{n} \Big( e_{i} \langle v , \psi+\langle \psi, a \rangle a\rangle  \Big)^{2}=
\sum_{i=1}^{n} \Big( \langle v , \overline{\nabla}_{e_{i}} \psi\rangle + \langle \overline{\nabla}_{e_{i}}\psi, a \rangle \langle a , v \rangle \Big)^{2}
$$
$$\qquad \qquad \qquad \quad \quad =\sum_{i=1}^{n} \Big( \langle v , e_{i} \rangle + \langle e_{i}, a \rangle \langle a , v \rangle \Big)^{2}=\| v^{T}\|^2 + \| a^{T}\| ^2 \langle a , v \rangle^{2} + 2 \langle a , v \rangle \langle v^{T}, a^{T} \rangle,$$ 
for every $v\in \L^m$.
Therefore, this formula gives
\begin{equation}\label{f90}
\mathrm{trace}(A_{Q_{1}})=n- \| a^{T}\| ^2 +2\| a^{T}\| ^2 =n +\| a^{T}\| ^2.
\end{equation}

\begin{lemma}\label{lemma4}
For every unit timelike vector $a\in \L^m$, the first eigenvalue $\lambda_1$ of the Laplace operator for a compact $n$-dimensional spacelike submanifold $M$ in Lorentz-Minkowski spacetime $\L^m$ with gravity center located at the origin satisfies
\begin{equation}\label{f100}
n\mathrm{Vol}(M)+\int_{ M} \| a^{T}\| ^2 dV \geq \lambda_{1}\int_{M}\Big( \| \psi \|^2 + \langle a, \psi\rangle^2 \Big)dV
\end{equation}
The equality holds if and only if $\Delta \psi + \lambda_{1} \psi=-\langle a, \Delta \psi+ \lambda_{1}\psi\rangle a$.
\end{lemma}
\begin{proof}
The vector field $\psi_{a}$ is a $\lambda_{1}$-test vector field. Hence the inequality (\ref{f100}) is a direct consequence from (\ref{f90}) and Lemma \ref{main-lemma}. The equality holds in (\ref{f100})
if and  only if $\Delta \psi_{a} +\lambda_{1}\psi_{a}=0$, or in an equivalent way $\Delta \psi + \lambda_{1} \psi=-\langle a, \Delta \psi+ \lambda_{1}\psi\rangle a$.
\end{proof}

\begin{remark}
{\rm If we have $\psi=\psi_a$ for some $a$, i.e., $\psi(M)\subset \E^{m-1}_a$, then conclusions in Lemma \ref{lemma2030} and \ref{lemma4}
are the same 
$$
\lambda_1 \int_{M} \| \psi \|^2\, dV \leq n\mathrm{Vol}(M),
$$
which gives the main Lemma in \cite{Re}.
}
\end{remark}

We assume one more time that the center of gravity of the compact spacelike immersion $\psi: M^{n}\to \L^{m}$ is located at the origin. Then, the Minimum Principle (\ref{minimum}) implies that the symmetric bilinear form on $\L^m$ defined by
$$
Q(v,w):=\int_{x\in M}\langle \nabla F_{v}(x), \nabla F_{w}(x) \rangle\, dV- \lambda_{1}\int_{x\in M}F_{v}(x)F_{w}(x)\, dV
$$
is positive semi-definite where $F_{v}(x)=\langle v , \psi (x)\rangle$. Therefore, for a vector $v\in \L^{m}$ the conditions $Q(v,v)=0$ and $Q(v,w)=0$ for all $w\in \L^m$ are equivalent.

The next result provides a sufficient condition in order to assure that a compact spacelike submanifold satisfies inequality (\ref{reilly}).

\begin{proposition}\label{elle}
Given a compact $n$-dimensional spacelike submanifold $M$ in Lorentz-Minkowski spacetime $\L^m$ with gravity center located at the origin,  assume there exists a causal vector $\ell \in \L^{m}$ $($i.e., $\|\ell\|^2\leq 0$ and $\ell \neq 0$ $)$ such that $Q(\ell, \ell)=0.$ Then, the first eigenvalue $\lambda_1$ of the Laplace operator of $M$ satisfies
$$
\lambda_{1}\leq n\frac{
\displaystyle{\int}_{ M} \|  \mathbf{H}\| ^{2}\, dV}{\mathrm{Vol}(M)}.
$$
The equality holds if and only if $\| \Delta \psi +\lambda_{1}\psi \|^2=0$ and  $\lambda_{1}\int_{M} \| \psi\|^2 \, dV = n \mathrm{Vol}(M)$.
\end{proposition}
\begin{proof}
From the Beltrami equation (\ref{beltrami}), the assumption $Q(\ell, \ell)=0$ is equivalent to \linebreak $\langle \ell , n\mathbf{H}+ \lambda_{1}\psi \rangle =0$, and therefore $\| n\mathbf{H}+ \lambda_{1}\psi \|^{2}\geq 0$ holds. 

On the other hand, the Minkowski formula (\ref{1040})
gives 
\begin{equation}\label{1050}
\int_{M}\| n\mathbf{H}+ \lambda_{1}\psi \|^{2}\, dV =n^2\int_{M}\|  \mathbf{H}\| ^{2}\, dV + \lambda_{1}^{2}\int_{M}\|  \psi\| ^{2}\, dV- 2 \lambda_{1}n \mathrm{Vol}(M) \geq 0.
\end{equation}
Now, let us fix $a\in \L^m$, unit timelike vector, such that $\langle a, \ell \rangle < 0$ and define the sequence of timelike vectors
$
w_{k}:=\ell + \frac{1}{k}a.
$
We are now ready to apply Lemma \ref{lemma4} for every \linebreak $a_{k}:=\frac{1}{\sqrt{- \| w_{k} \|^{2}}}w_{k}$ to obtain
$$
n\mathrm{Vol}(M)+\int_{ M} \| a_{k}^{T}\| ^2 dV \geq \lambda_{1}\int_{M}\Big( \| \psi \|^2 + \langle a_{k}, \psi\rangle^2 \Big)dV.
$$
In other words, we have
$$
Q(a_{k}, a_{k}) \geq \lambda_{1}\int_{M} \| \psi \|^2 \, dV - n \mathrm{Vol}(M) \quad \mathrm{for\,\,\, all} \quad k\geq 1.
$$
We claim that $\lim_{k\to \infty}Q(a_{k}, a_{k})=0$. In fact, a strightfoward computation shows
$$
\lim_{k\to \infty}Q(a_{k}, a_{k})=-\lim_{k\to \infty}\frac{Q(a,a)}{k^{2} \| \ell  \|^{2} + 2k\langle a, \ell \rangle -1}=0.
$$
Thus, we have $\lambda_{1}\int_{M} \| \psi \|^2 \, dV \leq n \mathrm{Vol}(M)$ and then from (\ref{1050})
the following string of inequalities
$$
0 \leq \int_{M}\| n\mathbf{H}+ \lambda_{1}\psi \|^{2}\, dV \leq n^2\int_{M}\|  \mathbf{H}\| ^{2}\, dV -  \lambda_{1}n \mathrm{Vol}(M).
$$ 
It remains only to show the equality case. Assume $n\int_{M}\|  \mathbf{H}\| ^{2}\, dV =  \lambda_{1} \mathrm{Vol}(M)$. Then it is not difficult to show that  $\| n\mathbf{H}+ \lambda_{1}\psi \|^2=0$ and $\lambda_{1}\int_{M} \| \psi \|^2 \, dV = n \mathrm{Vol}(M)$. The converse follows in a similar way.
\end{proof}

There are two natural families of compact spacelike submanifolds satisfying the assumption in Proposition  \ref{elle}. Namely, submanifolds which factor through spacelike hyperplanes and submanifolds through lightlike hyperplanes. Although the two following Corollaries are a direct consequence of formula (\ref{reilly}), as announced in the Introduction, we derive now them from Proposition \ref{elle} for the sake of completeness. Note that Corollary \ref{gra} now includes a characterization of the equality condition.

\begin{corollary}
The first eigenvalue $\lambda_1$ of the Laplace operator for a compact $n$-dimensional spacelike submanifold $M$ in Lorentz-Minkowski spacetime $\L^m$, which factors through a spacelike hyperplane $\Pi$, satisfies
\begin{equation}\label{f110}
\lambda_{1}\leq n\, \frac{
\displaystyle{\int}_{ M} \|  \mathbf{H}\| ^{2} dV}{\mathrm{Vol}(M)}.
\end{equation}
The equality holds if and only if $M$ is a minimal submanifold in some hypersphere in $\Pi$ of radius $\sqrt{n/\lambda_1}$ in $\Pi$.
\end{corollary}
\begin{proof}
Without loss of generality, may be assumed that the center of gravity of $\psi$ is located at the origin. Let us consider a normal unit timelike vector $a \in \L^m$ to $\Pi$.
The inequality (\ref{f110}) is a consequence of Proposition \ref{elle} applied to the vector $\ell=a$.
For the equality case, recall that $ \langle n\mathbf{H}+ \lambda_{1}\psi, \ell \rangle =0$ and $ \ell$ is timelike. Therofore, we derive that $n\mathbf{H}+ \lambda_{1}\psi= \Delta \psi +\lambda_{1}\psi =0$ and the classical Takahashi result \cite{Taka} can be applied to get that $M$ is immersed, with zero mean curvature, in some hypersphere of radius $\sqrt{n/\lambda_1}$ in $\Pi$.
\end{proof}

\begin{corollary}\label{gra}
The first eigenvalue $\lambda_1$ of the Laplace operator for a compact $n$-dimensional spacelike submanifold $M$ in Lorentz-Minkowski spacetime $\L^m$, which factors through a lightlike hyperplane $\Pi$ with lightlike nomal vector $\ell\in \L^m$, satisfies
\begin{equation}\nonumber
\lambda_{1}\leq n\, \frac{
\displaystyle{\int}_{ M} \|  \mathbf{H}\| ^{2} dV}{\mathrm{Vol}(M)}.
\end{equation}
The equality holds if and only if there is a function $\beta_{\ell}$ such that $ \Delta \psi +\lambda_{1}\psi =\beta_{\ell} \ell$ and  $\lambda_{1}\int_{M} \| \psi \|^2 \, dV = n \mathrm{Vol}(M)$.
\end{corollary}
\begin{proof}
We may also assume that the center of gravity of $\psi$ is located at the origin. 
The inequality result is a consequence of Proposition \ref{elle} applied to the vector $\ell\in \L^m$.
For the equality case, recall that $ \langle n\mathbf{H}+ \lambda_{1}\psi, \ell \rangle =0$ and $ \ell$ is lightlike. Therofore, in the equality case, we derive that $n\mathbf{H}+ \lambda_{1}\psi= \Delta \psi +\lambda_{1}\psi =\beta_{\ell}\ell$ for some $\beta_{\ell}\in C^{\infty}(M)$. 
\end{proof}

\begin{remark}
{\rm At this point it is natural to wonder if inequality (\ref{reilly}) holds only for compact spacelike submanifolds of $\L^m$ through a spacelike or lightlike hyperplane. As was mentioned in the Introduction, every compact spacelike surface in $\L^4$ through a lightlike cone satisfies the inequality (\ref{reilly}). Therefore, the answer is negative. Even more, for a compact spacelike surface $M^2$ through a lightcone in $\L^4$ the following conditions are equivalent \cite[ Section 4 and Theorem 5.4]{PaRo}: (1) $M^2$ has constant Gauss curvature, (2) $M^2$ is totally umbilical in $\L^4$ and (3) $M^2$ factors through a spacelike hyperplane.
Thus, a compact spacelike surface in $\L^4$ through a lightlike cone and not totally umbilical satisfies (\ref{reilly}) and does not factorizes through any spacelike or lightlike hyperplane.}
\end{remark}

Now we are in position to prove Proposition \ref{oldprincipal} and Theorem \ref{reprincipal}. These results can be thought as suitable alternatives to the Reilly inequality (\ref{reilly}) for compact spacelike submanifolds in $\L^m$. 
Given such a submanifold, we may assume, performing certain translation if it is necessary, its gravity center is located at the origin of $\L^m$. The original immersion and the translated one have the same mean curvature vector fields. Moreover, for every $v\in \L^{m}$, the corresponding tangent parts $v^{T}=\nabla \langle v , \psi\rangle$ also agree.
Therefore, all the quantities appearing in the following inequalities $\mathrm{(E)}$ and $\mathrm{(E^{*})}$ are independent of the
choice of origin. 

\begin{proposition}\label{oldprincipal}
For each unit timelike vector $a\in \L^{m}$, the first eigenvalue $\lambda_1$ of the Laplace operator of a compact $n$-dimensional spacelike submanifold $M$ in Lorentz-Minkowski spacetime $\L^m$, $m\geq n+2$, satisfies
$$
\hspace*{-3cm}\mathrm{(E)} \hspace*{3cm}\lambda_1 \leq n\,\frac{\displaystyle{\int}_{M}\|\mathbf{H}_a\|^2 \,dV}{\mathrm{Vol}(M) +\frac{1}{n} \displaystyle{\int}_{M}\| a^{T}\|^2\, dV}\,,
$$
where $\mathbf{H}_{a}=\mathbf{H} + \langle \mathbf{H}, a \rangle a$ is the orthogonal projection of the mean curvature vector field $\mathbf{H}$ on the spacelike hyperplane $a^{\perp}$, and $a^{T}\in \mathfrak{X}(M)$ is, at any point $p\in M$, the orthogonal projection of $a$ on $T_{p}M$. The equality in $\mathrm{(E)}$ holds if and only if there exists $\mu_{a}\in C^{\infty}(M)$ such that
$
\Delta \widehat{\psi} + \lambda_{1} \widehat{\psi}=\mu_{a} a$
where $\widehat{\psi}$ is given by means of formula $(${\rm \ref{hat}}$)$.
\end{proposition}
\begin{proof}
Without loss of generality, assume the gravity center of $\psi$ is located at the origin.
Thus, Lemma \ref{lemma4} gives
$$
n\mathrm{Vol}(M)+\int_{ M} \| a^{T}\| ^2 dV \geq \lambda_{1}\int_{M}\| {\psi}_{a}\|^2 dV\,,
$$
where $\psi_{a}$ is given in (\ref{f2010}). The equality holds for some unit timelike vector $a\in \L^m$ if and only if  $\Delta \psi+ \lambda_{1} \psi=\mu_{a} a$
where $\mu_{a}=-\langle a, \Delta \psi + \lambda_{1} \psi \rangle$.

Let us consider the orthogonal projection
\begin{equation}\label{1030}
\mathbf{H}_{a}= \mathbf{H}+\langle \mathbf{H}, a\rangle a\in \E^{m-1}_{a},
\end{equation}
where $\E^{m-1}_{a}=a^{\perp}$. Taking into account that $\mathbf{H}$ is a $\lambda_1$-test vector field, it is not difficult to deduce that $\mathbf{H}_{a}$ does not vanish identically and therefore $ \int_{M}\| \mathbf{H}_{a}\|^2 dV >0$.

Now the Cauchy-Schwarz inequalities for integrals and vectors on $\E^{m-1}_{a}$ give the following string of inequalities
$$
\Big(n\mathrm{Vol}(M)+\int_{ M} \| a^{T}\| ^2 dV\Big) \int_{M}\| \mathbf{H}_{a}\|^2 dV \geq
$$
\begin{equation}\label{formula1} 
\lambda_{1}\int_{M}\| \psi_{a}\|^2 dV \int_{M}\| \mathbf{H}_{a}\|^2 dV\geq \lambda_{1}\Big(\int_{M}\langle \psi_{a},  \mathbf{H}_{a}\rangle \Big)^{2}.
\end{equation}

As a direct application of Minkowski formula (\ref{1040}), we have
\begin{equation}\label{f3010}
\int_{M}\Big(1+\langle \psi_{a}, \mathbf{H}_{a}\rangle- \langle \psi,a \rangle  \langle \mathbf{H}, a\rangle\Big)dV=0,
\end{equation}
and, on the other hand, using again Beltrami equation (\ref{beltrami}),
$$\triangle  \langle \psi,a \rangle^{2}=2 n \langle \psi,a \rangle \langle \mathbf{H},a \rangle + 2\| a^{T}\| ^2$$ holds. Therefore, using the divergence Theorem and previous equation in (\ref{f3010}), we have
\begin{equation}\label{formula2} 
\int_{M}\langle \psi_{a},  \mathbf{H}_{a}\rangle=-\mathrm{Vol}(M)-\frac{1}{n}\int_{M}\| a^{T}\| ^2 dV.
\end{equation}

Finally, from (\ref{formula1}) and (\ref{formula2}) we obtain the desired inequality $\mathrm{(E)}$. The equality holds in (E) if and only if the equality holds
in Lemma \ref{lemma4} and this ends the proof.
\end{proof}

\begin{remark}
{\rm In more geometric terms, the equality condition in (E), via Beltrami equation (\ref{beltrami}), gives
$$
n\mathbf{H}=\mu_{a} a- \lambda_{1}\widehat{\psi}.
$$
Therefore, for $v\in T_{x}M$, we get
$
\overline{\nabla}_{v}\mathbf{H}=\frac{1}{n}v(\mu_{a} )a- \frac{\lambda_{1}}{n}v,
$
and thus, taking in mind Weingarten formula, we have
$$
A_{\mathbf{H}}=\frac{\lambda_{1}}{n}\mathrm{Id}-\frac{1}{n}d\mu_{a} \otimes a^{T} \quad \mathrm{and} \quad \nabla^{\perp}\mathbf{H}=\frac{1}{n}d\mu_{a} \otimes a^{N},
$$
where $a^{N}\in \mathfrak{X}^{\perp}(M)$ is, at any point $p\in M$, the orthogonal projection of $a$ on $T_{p}^{\perp}M$.
Hence, we compute
$$
n\| \mathbf{H}\|^{2}=\mathrm{trace}(A_{\mathbf{H}})=\lambda_{1}- \frac{1}{n}\langle \nabla\mu_{a}, a^{T} \rangle, \quad \nabla \| \mathbf{H}\|^2 = \frac{2}{n}\langle a^{N}, \mathbf{H}\rangle \nabla \mu_{a}. 
$$
and thus
$$
\Big{\|} \widehat{\psi}- \frac{\mu_{a}}{\lambda_{1}}\,a \Big{\|}^2 =\frac{n}{\lambda_{1}}- \frac{1}{\lambda_{1}^2}\langle \nabla\mu_{a}, a^{T} \rangle.$$
Now, observe that we have $\int_{M}\mu_{a}\, dV=0$ from the definition of $\mu_a$. 
Therefore, $\mu_a$ is constant if and only if 
$\mu_a =0$, in this case we have 
$\| \widehat{\psi}\|^2 =n/\lambda_{1}$ and $M$ is contained in a De Sitter space of radius $\sqrt{n/\lambda_{1}}$ with zero mean curvature vector field.
}
\end{remark}

\begin{remark}
{\rm In general, there is no unit timelike vector $a\in \L^m$
such that equality holds in $\mathrm{(E)}$ for a given compact spacelike submanifold. 
In fact, this is the case for the counter-example in Section 3, as we will explain now.

A direct computation from Beltrami equation (\ref{beltrami}) and formula (\ref{contraejeH}) shows
$$
\Delta \psi+ n\psi=\Big((1+n-t^2)\cosh(t)-nt\sinh(t),(1+n-t^2)\sinh(t)-nt\cosh(t),0, \cdots,0\Big),
$$
at any $(t,y)\in \S^n$.  We derive a contradiction as follows. Assume the equality condition for $\mathrm{(E)}$ is satisfied and let us write $\widehat{\psi}=\psi +b$ for a suitable $b\in \L^m$. The condition $\Delta \widehat{\psi} + n \widehat{\psi}=\mu_{a}a$ holds  for a unit timelike vector $a=(a_{1},... ,a_{m})\in \L^m$ if and only if
\begin{equation}\label{1000}
\Big((1+n-t^2)\cosh(t)-nt\sinh(t),(1+n-t^2)\sinh(t)-nt\cosh(t),0,\cdots,0\Big)+n b=\mu_{a}a.
\end{equation}
Taking into account that $a_1 \neq 0$, we get 
$$\mu_a =\frac{1}{a_{1}}\Big((1+n-t^2)\cosh(t)-nt\sinh(t)+n b_{1}\Big), \quad (t,y)\in \S^n ,
$$
and also 
$$a_{2}\mu_{a}=(1+n-t^2)\sinh(t)-nt\cosh(t)+n b_{2}.
$$
Therefore, a direct computation shows that the following function
$$
\Big[\frac{a_{2}}{a_{1}}(1+n-t^2)+ nt\Big]\cosh(t)-\Big[\frac{a_{2}}{a_{1}}nt+(1+n-t^2)\Big]\sinh(t), \quad t\in[-1,+1]
$$
must be constant, which is clearly a contradiction.
}
\end{remark}

As a direct consequence of Proposition \ref{oldprincipal} we arrive to the main result of this paper

\begin{theorem}\label{reprincipal}
For each unit timelike vector $a\in \L^{m}$, the first eigenvalue $\lambda_1$ of the Laplace operator for a compact $n$-dimensional spacelike submanifold $M$ in Lorentz-Minkowski spacetime $\L^m$, $m\geq n+2$, satisfies
$$
\hspace*{-5cm}\mathrm{(E^{*})} \hspace*{4cm}\lambda_1 \leq n\, \frac{\displaystyle{\int}_{M}\|\mathbf{H}_a\|^2 \,dV}{\mathrm{Vol}(M)}.\hspace*{-1cm}
$$
The equality holds if and only if  $M$ factors through a spacelike affine hyperplane $\Pi$ orthogonal to $a\in \L^m$ and $M$ is minimal in some hypersphere in $\Pi$ with radius $\sqrt{n/\lambda_1}$. 
\end{theorem}
\begin{proof}
Clearly the inequality (E) implies $(\mathrm{E}^{*})$. The equality holds in $(\mathrm{E}^{*})$ if and only if $a^{T}=\nabla \langle a, \psi \rangle=0$ and we have equality in (E). That is, $M$ factors through a spacelike affine hyperplane $\Pi$ orthogonal to $a\in \L^m$ and $\Delta \widehat{\psi} + \lambda_{1}\widehat{\psi}=0$. The classical Takahashi result \cite{Taka} ends the proof.
\end{proof}

\begin{remark}\label{end}
{\rm From each spacelike submanifold $\psi:M\to  \L^m$ and each unit timelike vector $a\in \L^m$, we can consider the immersion
$\psi_{a}:M \to \E^{m-1}_{a}$ given by $\psi_{a}(x):=\psi(x) +\langle a, \psi(x)\rangle a$ where $ \E^{m-1}_{a}=a^{\perp}$. Obviously, the Reilly inequality (\ref{reilly}) holds for $\psi_a$. Nevertheless, the metric induced via $\psi_a$ and $\psi$ are different, in general. In fact, a direct computation gives
$$
(\psi_{a})^{*}\langle\,\,,\,\, \rangle_{\E^{m-1}_{a}}=\psi^{*}\langle\,\,,\,\, \rangle_{\L^m}+ d\psi_{1}\otimes d\psi_{1},
$$
where $\psi_{1}=-\langle a, \psi \rangle$ and of course, the mean curvature vector field  $\mathbf{H}_{\psi_a}$ corresponding to $\psi_a$ and $\mathbf{H}_a$ are also different, in general. 
Consequently, one cannot think that $\mathrm{(E^{*})}$ can be derived from (\ref{reilly}).}
\end{remark}

This paper concludes with the following direct application of Theorem \ref{reprincipal}.
\begin{corollary}\label{noche}
If the first eigenvalue $\lambda_1$ of the Laplace operator of a compact spacelike $n$-dimensional spacelike submanifold  $M$ in $\L^{n+2}$ satisfies $\lambda_{1}=n\|\mathbf{H}_{a}\|^2$ for some unit timelike vector $a\in \L^m$. Then $M$ is contained in a spacelike affine hyperplane orthogonal to $a\in \L^m$ as a round $n$-sphere of radius $1/\|\mathbf{H}_{a}\|$.
\end{corollary}

\begin{remark}
{\rm The families of inequalities (E) and $(\mathrm{E}^{*})$ are parametrized on the set $\H^{m-1}$ of unit timelike vectors in $\L^m$. Thus, for a compact $n$-dimensional spacelike submanifold $M$ in Lorentz-Minkowski spacetime $\L^m$,
 the first eigenvalue $\lambda_1$ of the Laplace operator of $M$ satisfies
$$
\lambda_{1}\leq \inf_{a\in \H^{m-1}}\,n\,\frac{\displaystyle{\int}_{M}\Big(\|\mathbf{H}\|^2 + \langle \mathbf{H} , a \rangle ^2\Big)\,dV}{\mathrm{Vol}(M) +\frac{1}{n} \displaystyle{\int}_{M}\| a^{T}\|^2\, dV}.
$$
A similar inequality is obtained from $(\mathrm{E}^{*})$.

}
\end{remark}

\end{document}